\documentclass[11pt,a4paper]{article}
\usepackage[T1]{fontenc}
\usepackage{lmodern,microtype}
\usepackage[margin=27mm]{geometry}
\usepackage{amsmath,amssymb,amsthm,mathtools}
\usepackage{enumitem,xcolor}
\usepackage[colorlinks=true,linkcolor=blue!50!black,citecolor=blue!50!black,urlcolor=blue!50!black]{hyperref}
\hypersetup{pdftitle={Dynamical Geometry of Principal Bundle Constrained Systems: Compatible Pairs, Contact Degeneration, and Adapted Metrics},pdfauthor={Dongzhe Zheng}}
\numberwithin{equation}{section}
\newtheorem{theorem}{Theorem}[section]
\newtheorem{proposition}[theorem]{Proposition}
\newtheorem{corollary}[theorem]{Corollary}
\newtheorem{lemma}[theorem]{Lemma}
\theoremstyle{definition}
\newtheorem{definition}[theorem]{Definition}
\newtheorem{example}[theorem]{Example}
\theoremstyle{remark}
\newtheorem{remark}[theorem]{Remark}
\newcommand{\g}{\mathfrak g}
\newcommand{\h}{\mathfrak h}
\newcommand{\z}{\mathfrak z}
\newcommand{\D}{\xi}
\newcommand{\Horz}{\mathcal H}
\newcommand{\Ad}{\operatorname{Ad}}
\newcommand{\Hol}{\operatorname{Hol}}
\newcommand{\rank}{\operatorname{rank}}
\newcommand{\Char}{\operatorname{Char}}
\newcommand{\Lie}{\mathcal L}
\newcommand{\R}{\mathbb R}
\newcommand{\Z}{\mathbb Z}
\newcommand{\SU}{\mathrm{SU}}
\newcommand{\SO}{\mathrm{SO}}

\newcommand{\dd}{\mathrm d}
\DeclareMathOperator{\ad}{ad}

\setlist[enumerate]{itemsep=3pt,topsep=5pt}

\title{Dynamical Geometry of Principal Bundle Constrained Systems:\\Compatible Pairs, Contact Degeneration, and Adapted Metrics}
\author{Dongzhe Zheng}
\date{}

\begin{document}
\maketitle
\begin{abstract}
We study invariant codimension-one constraints on principal bundles through
compatible pairs: a constraint distribution and a nonzero coadjoint field
parallel for a principal connection. Pairing the field with the connection
gives an invariant one-form whose Levi form separates a horizontal curvature
contribution from a vertical coadjoint-orbit contribution. This decomposition
yields criteria for integrability, contactness, and characteristic reduction;
holonomy and stabilizer reductions describe global existence. For evolving
compatible pairs, we identify the mixed-curvature obstruction to compatibility
and prove that connection transport preserves the Levi geometry.
For initially contact data over a closed base of dimension $2n$, we establish
matching bounds for the quadratic $H^{n+1}$ cost of contact degeneration:
making the paired curvature vanish on a Darboux ball of radius $r$ in time $T$
costs an amount comparable to $[T\log(R_*/r)]^{-1}$. The constructed paths
remain contact before $T$, preserve the curvature class, and force the
$L^\infty$ norm of every transporting velocity gradient to grow at least
as $1/[2(T-t)]$ on the
collapsing region in Darboux coordinates. On a closed three-manifold,
a contact form preserved by a locally free circle action admits an invariant
adapted metric with any prescribed positive curl eigenvalue and unit circle
generator. We parametrize all such metrics and prove that their space is
contractible. Along the circle-bundle degeneration paths, every continuous
tensor limit of normalized adapted metrics is degenerate above the collapsing
region.
\end{abstract}

\section{Introduction}

A cooriented codimension-one velocity constraint on a principal bundle $P\to M$ is
specified by a nowhere-zero one-form $\alpha$: an admissible velocity
$V\in T_pP$ satisfies $\alpha_p(V)=0$. We study invariant constraints
obtained by pairing a principal connection $A$ with a nonzero equivariant
coadjoint field $\lambda$. The distribution $\xi$ and the field $\lambda$
form a compatible pair for $A$ when
\[
 \xi=\ker\alpha,\qquad \alpha=\langle\lambda,A\rangle,\qquad
 \dd\lambda+\ad_A^*\lambda=0.
\]
For a compact structure group, an invariant inner product $B$ identifies
$\lambda=B(u,\cdot)$ with a parallel adjoint field $u$, and
$\alpha=B(u,A)$.

The differential of $\alpha$ records both the curvature on the base and
the Lie brackets along the fibers. These contributions determine the
integrability, contactness, and characteristic leaves of the constraint
distribution. For evolving connections, mixed spacetime curvature controls
the persistence of parallelism, while a covariant transport equation carries
the constraint by bundle automorphisms. We analyze this constraint--curvature
coupling, the cost of contact degeneration, and the normalization of adapted
metrics for invariant contact forms in dimension three.

The first main result concerns a fixed principal bundle over a closed
$2n$-manifold and a fixed nonzero adjoint field $u$. Starting from a
connection $A_0$ for which $\dd_{A_0}u=0$ and $B(u,A_0)$ is contact, we
consider smooth paths $A_t$ satisfying $\dd_{A_t}u=0$. Write
$\pi^*\beta_t=B(u,F_{A_t})$. For a fixed Darboux ball $B_{R_*}$ and
$0<r<R_*$, impose the terminal condition $\beta_T|_{B_r}=0$ and require
$B(u,A_t)$ to remain contact for $t<T$. The infimum of
$\int_0^T\|\dot A_t\|_{H^{n+1}}^2\,\dd t$ over these paths satisfies
\[
 \mathcal C_T(r)\asymp\frac{1}{T\log(R_*/r)}
 \qquad(r\to0).
\]
The constants are independent of $r$ and $T$. Theorem~\ref{thm:critical}
gives the construction and the matching lower bound. It also identifies
$H^{n+1}$ as the threshold for openness of the contact condition among
smooth connections preserving $u$. Proposition~\ref{prop:strain} shows
that every vector field transporting the constructed paired curvature
satisfies $\|\nabla v_t\|_{L^\infty}\geq1/[2(T-t)]$.

The second main result fixes a contact form $\alpha$ on a closed
three-manifold $N$ and a locally free circle action preserving it. If $X$
is the circle generator, then for every $\kappa>0$ there is an invariant
Riemannian metric $g$ such that
\[
 *_g\dd\alpha=\kappa\alpha,\qquad g(X,X)=1.
\]
Theorem~\ref{thm:metric} gives an explicit formula which is smooth along
the locus $\alpha(X)=0$ and at exceptional circle orbits.
Proposition~\ref{prop:metric-space} describes all metrics satisfying these
conditions and proves contractibility of their space. On circle bundles
over surfaces, the contact degeneration paths admit these normalized
metrics at every preterminal time. Corollary~\ref{cor:metric-degeneration}
shows that every continuous tensor limit of such a metric family is
degenerate above the region where the terminal curvature vanishes.

The geometric background comes from holonomy, contact coupling, and
prequantization. Parallel sections correspond to vectors fixed by full
holonomy; Ambrose and Singer describe restricted holonomy in terms of
curvature~\cite{AS}. Lerman's contact fiber-bundle construction
\cite[Proposition 3.1 and Theorem 3.3]{Lerman} contains the contact
coupling considered here, and the circle case is the Boothby--Wang
construction~\cite{BW}. We give the Levi decomposition, stabilizer
reduction, and characteristic quotient in the conventions used throughout
the paper.

For closed two-forms, Swanson and Chicone~\cite[Theorem 2.1]{SC} establish
openness of a diffeomorphism orbit in a fixed cohomology class in the
$H^s$ topology for $s>\dim M/2$. The critical logarithmic estimates in
Sobolev potential theory~\cite{AH} govern the endpoint studied here.
Our construction realizes the corresponding scale-dependent cost by
exact curvature deformations that remain symplectic before their terminal
time. The associated transport equation is a bundle version of Moser's
deformation argument~\cite{Moser}.

For invariant contact forms in dimension three, Question 2.7 in the
arXiv revision of~\cite{Kom} asks whether the curl coefficient can be constant
when the circle generator is unit Killing. Theorem~\ref{thm:metric}
answers this question. Its proof chooses the remaining scale in the
weakly compatible metric family of Etnyre, Komendarczyk, and
Massot~\cite[Proposition 2.1(3)]{EKM}, while fixing $\alpha$, $X$, and the
prescribed eigenvalue $\kappa$.

Section~\ref{sec:rank} develops the Levi decomposition, and
Section~\ref{sec:critical} proves the cost estimate. Sections
\ref{sec:holonomy}--\ref{sec:examples} describe holonomy, global
realization, characteristic quotients, and examples.
Section~\ref{sec:dynamics} treats compatible evolution and bundle
transport. Section~\ref{sec:metrics} establishes the metric
normalization theorem and its application to circle bundles.

\section{The Levi form and its rank}
\label{sec:rank}

All manifolds are smooth, finite dimensional, Hausdorff, and second countable.
Let $M$ be connected, let $G$ be a compact connected Lie group of positive
dimension, and let $\pi:P\to M$ be a right principal $G$-bundle. Fix a positive
definite $\Ad$-invariant inner product $B$ on $\g$. For $a\in\g$, write
$a^\#_p=\left.\frac{\dd}{\dd t}\right|_0 p\exp(ta)$.
For a smooth function $\zeta:P\to\mathfrak g$, define
$\zeta^\#_p=(\zeta(p))^\#_p$; equivariance of $\zeta$ makes this vertical
vector field $G$-invariant.

A connection $A\in\Omega^1(P;\g)$ satisfies
$A(a^\#)=a$ and $R_g^*A=\Ad_{g^{-1}}A$. Its curvature and covariant derivative
are
\begin{equation}\label{eq:conventions}
 F_A=\dd A+\tfrac12[A\wedge A],\qquad
 \dd_Au=\dd u+[A,u],\qquad
 \tfrac12[A\wedge A](U,V)=[A(U),A(V)].
\end{equation}
An adjoint section is represented by an equivariant function
$u:P\to\g$, $u(pg)=\Ad_{g^{-1}}u(p)$. A pair $(A,u)$ is called parallel when
\begin{equation}\label{eq:parallel}
 \dd_Au=0,\qquad u\not\equiv0.
\end{equation}
The norm of $u$ is a positive constant: indeed,
$\dd B(u,u)=-2B([A,u],u)=0$. Thus
\begin{equation}\label{eq:alpha}
 \alpha=B(u,A),\qquad \D=\ker\alpha
\end{equation}
are respectively a nowhere-zero invariant one-form and a smooth cooriented
hyperplane distribution. On a $(2m+1)$-manifold, a one-form is contact
when $\alpha\wedge(\dd\alpha)^m$ is nowhere zero; equivalently,
$\dd\alpha$ is nondegenerate on $\ker\alpha$.
A closed two-form is symplectic when it is nondegenerate. The corresponding coadjoint field is
$\lambda=B(u,\cdot)$; with
$(\ad_a^*\lambda)(b)=-\lambda([a,b])$, equation~\eqref{eq:parallel} is
$\dd\lambda+\ad_A^*\lambda=0$.

For a fixed connection $A$, a pair $(\xi,\lambda)$ consisting of an
invariant hyperplane distribution and a nowhere-zero equivariant
coadjoint field is called \emph{compatible with $A$} when
\[
 \xi=\ker\langle\lambda,A\rangle,\qquad
 \dd\lambda+\ad_A^*\lambda=0.
\]
The inner product $B$ identifies these data with the parallel pairs
$(A,u)$ in~\eqref{eq:parallel}. Thus~\eqref{eq:alpha} specifies the
constraint distribution associated with a compatible pair.

The scalar two-form $B(u,F_A)$ is horizontal and invariant, so it descends to a
unique $\beta\in\Omega^2(M)$:
\begin{equation}\label{eq:beta}
 \pi^*\beta=B(u,F_A).
\end{equation}
For $p\in P$, put $x=\pi(p)$, $\mu=u(p)$, and
$\g_\mu=\{a\in\g:[a,\mu]=0\}$. The adjoint orbit is
$\mathcal O_\mu=\Ad_G\mu$.

We use $\dd\alpha|_\D$ as a scalar representative of the Levi form.
With the quotient $TP/\D\cong\R$ induced by $\alpha$, the bracket convention
is $\alpha([U,V])=-\dd\alpha(U,V)$ for sections $U,V$ of $\D$.
Its characteristic space is
\[
 \Char(\D)_p=\{U\in\D_p:\dd\alpha_p(U,V)=0\text{ for all }V\in\D_p\}.
\]

\begin{theorem}[Levi decomposition]\label{thm:rank}
Under the preceding hypotheses, $\dd\beta=0$, and
\begin{equation}\label{eq:dalpha}
 \dd\alpha=B(u,F_A)+\tfrac12B(u,[A\wedge A]).
\end{equation}
Writing $\Horz=\ker A$, there is a direct sum
\begin{equation}\label{eq:splitD}
 \D_p=\Horz_p\oplus(\mu^\perp)^\#_p.
\end{equation}
For $v,w\in T_xM$ and $a,b\in\mu^\perp$,
\begin{equation}\label{eq:leviblock}
 \dd\alpha_p(v^H+a^\#,w^H+b^\#)
 =\beta_x(v,w)+B(\mu,[a,b]).
\end{equation}
In particular,
\begin{align}
 \Char(\D)_p
 &= (\ker\beta_x)^H\oplus(\g_\mu\cap\mu^\perp)^\#_p,
 \label{eq:char}\\
 \rank(\dd\alpha|_{\D_p})
 &=\rank\beta_x+\dim\mathcal O_\mu.
 \label{eq:rank}
\end{align}
\end{theorem}

\begin{proof}
The Bianchi identity and parallelism give
$\dd B(u,F_A)=B(\dd_Au\wedge F_A)+B(u,\dd_AF_A)=0$.
Pullback by a surjective submersion is injective on forms, so $\dd\beta=0$.

To compute $\dd\alpha$, write $a=A(U)$ and $b=A(V)$. Differentiating the
pairing and using $\dd u=-[A,u]$ yields
\begin{align*}
 \dd\alpha(U,V)
 &=-B([a,u],b)+B([b,u],a)+B(u,\dd A(U,V))\\
 &=2B(u,[a,b])+B(u,\dd A(U,V))\\
 &=B(u,F_A(U,V))+B(u,[a,b]).
\end{align*}
This proves~\eqref{eq:dalpha}. Decomposition~\eqref{eq:splitD} follows from
$T_pP=\Horz_p\oplus\g^\#_p$ and $\alpha(a^\#)=B(\mu,a)$.
Horizontality of $F_A$ gives~\eqref{eq:leviblock}.

For the vertical block, invariance gives
$B(\mu,[a,b])=B([\mu,a],b)$. If $a\in\mu^\perp$ lies in its radical,
$[\mu,a]$ is orthogonal to $\mu^\perp$, hence lies in $\R\mu$.
It is also orthogonal to $\mu$, and therefore vanishes. Conversely, any
$a\in\g_\mu\cap\mu^\perp$ lies in the radical. Since
$\mu\in\g_\mu$ and $B(\mu,\mu)>0$, this radical has dimension
$\dim\g_\mu-1$. The vertical rank is consequently
$(\dim\g-1)-(\dim\g_\mu-1)=\dim\mathcal O_\mu$.
Adding the ranks of the two blocks proves the result.
\end{proof}

\begin{corollary}[Three geometric regimes]\label{cor:regimes}
The distribution $\D$ has the following properties.
\begin{enumerate}
\item It is Frobenius integrable if and only if $u(p)\in\z(\g)$ and
      $\beta=0$.
\item At $p$, its sections and their brackets span $T_pP$ if and only if
      $u(p)\notin\z(\g)$ or $\beta_{\pi(p)}\ne0$.
\item The form $\alpha$ is contact if and only if $\beta$ is symplectic and
      $\dim\g_{u(p)}=1$.
\end{enumerate}
\end{corollary}

\begin{proof}
Frobenius integrability is equivalent to vanishing of the Levi form. Formula
\eqref{eq:rank} then gives $\beta=0$ and $\dim\mathcal O_\mu=0$; since $G$
is connected, the latter is $\mu\in\z(\g)$. A hyperplane distribution
generates in two steps at a point exactly when its Levi form is nonzero there:
one nonzero bracket class spans the one-dimensional quotient. Finally,
contactness is equivalent to nondegeneracy of $\dd\alpha|_\D$, and
\eqref{eq:char} gives the two stated conditions. Horizontal parallel transport
and equivariance put all values of $u$ in a single adjoint orbit, so their
centrality and centralizer dimension are constant.
\end{proof}

For a compact semisimple Lie algebra, every nonzero $u$ is noncentral.
Thus every pair satisfying~\eqref{eq:parallel} produces a distribution that
generates in two steps everywhere, including when $F_A=0$.
Contactness is stronger. Since the centralizer of any element of a compact
Lie algebra contains a maximal torus, it has dimension at least $\rank G$.
Thus a parallel pair with $\alpha$ contact requires $\rank G=1$.

Also, $\dim(\D_p\cap\ker\dd\pi)=\dim G-1$. Consequently $\D$ is a
horizontal distribution only when $\dim G=1$, in which case $\D=\ker A$.
An Atiyah-sequence splitting specifies $\ker A$, whereas~\eqref{eq:alpha}
usually specifies a larger distribution.

\section{The logarithmic cost of contact degeneration}\label{sec:critical}

Let $M$ be closed and connected with $\dim M=2n\geq2$.
Fix a parallel pair $(A_0,u)$ such that $\alpha_0=B(u,A_0)$ is contact,
and write $c_u=B(u,u)>0$ and $\pi^*\beta_0=B(u,F_{A_0})$.
Corollary~\ref{cor:regimes} gives a symplectic form $\beta_0$ and
$\mathfrak g_{u(p)}=\mathbb Ru(p)$ at every $p\in P$.

Fix a smooth Riemannian metric $g_M$ on $M$. All Sobolev norms are
nonhomogeneous $L^2$ norms on tensor fields on the base, defined using
$g_M$, its Levi--Civita connection, $B$, and the reference connection
$A_0$ on $\operatorname{ad}P$. Connection differences and curvature
differences are respectively adjoint-valued one-forms and two-forms on $M$.
They induce relative topologies on the space of smooth connections.
The affine space of connections preserving $u$ is
\[
 \mathcal A_u=\{A\in\mathcal A(P):\dd_Au=0\},
\]
where $\mathcal A(P)$ denotes the space of smooth principal connections.
We regard a path as the connection $\mathbb A=A_t$ on $P\times[0,T]$,
with zero temporal component. Its mixed curvature is $E_t=\dot A_t$.
The deformation cost below uses this fixed gauge and the fixed background
norms. General temporal components are treated in Section~\ref{sec:dynamics}.

\begin{lemma}[Connections preserving the adjoint field]\label{lem:scalar-source}
The map
\begin{equation}\label{eq:scalar-source}
 b\longmapsto A_b=A_0+\frac{u}{c_u}\,\pi^*b
\end{equation}
is an affine bijection from $\Omega^1(M)$ onto $\mathcal A_u$. Moreover,
\begin{equation}\label{eq:scalar-source-curvature}
 F_{A_b}=F_{A_0}+\frac{u}{c_u}\pi^*(\dd b),\qquad
 \alpha_b=\alpha_0+\pi^*b,\qquad
 \beta_b=\beta_0+\dd b,
\end{equation}
and $F_{A_b}=(u/c_u)\pi^*\beta_b$.
\end{lemma}

\begin{proof}
The added term in~\eqref{eq:scalar-source} is horizontal and equivariant,
and it commutes with $u$, so it is a connection difference preserving
$\dd_Au=0$. Conversely, if $A'=A_0+a$ lies in $\mathcal A_u$, then
$[a,u]=0$. The one-dimensional centralizer condition gives
$a=(u/c_u)\pi^*b$, with $\pi^*b=B(u,a)$. The curvature formula follows from
$\dd_{A_0}u=0$ and the vanishing self-bracket of the line-valued perturbation.
Finally, $\dd_A^2u=[F_A,u]=0$ puts every curvature value in $\R u$, proving
the last assertion.
\end{proof}

Choose a Darboux chart containing the closure of a coordinate ball $B_{R_*}$, with
\begin{equation}\label{eq:critical-darboux}
 \beta_0=\omega_0=\sum_{j=1}^n\dd x_j\wedge\dd y_j,
 \qquad
 \lambda_{\mathrm{std}}=\frac12\sum_{j=1}^n(x_j\dd y_j-y_j\dd x_j),
 \qquad \dd\lambda_{\mathrm{std}}=\omega_0.
\end{equation}
Here $\rho=(\sum_j(x_j^2+y_j^2))^{1/2}$ is the Euclidean radial coordinate
and $B_r=\{\rho<r\}$. The chart is taken slightly larger than
$\overline{B_{R_*}}$. For $T>0$ and
$0<r<R_*$, let $\mathcal P_T(r)$ be the class of smooth paths
$A_t\in\mathcal A_u$, $0\leq t\leq T$, which start at $A_0$, have
$B(u,A_t)$ contact for $t<T$, and satisfy $\beta_T|_{B_r}=0$. Define
\begin{equation}\label{eq:critical-cost}
 \mathcal C_T(r)=
 \inf_{A_\bullet\in\mathcal P_T(r)}
       \int_0^T\|\dot A_t\|_{H^{n+1}}^2\,\dd t.
\end{equation}
The terminal condition removes the horizontal contribution to the
Levi form throughout the prescribed ball.

\begin{theorem}[Critical cost and sharp Sobolev threshold]\label{thm:critical}
There exist constants $0<r_0<R_*/e$ and $c_1,c_2>0$, depending on $(A_0,u)$, the chart,
and norms, such that for $0<r<r_0$ and every $T>0$,
\begin{equation}\label{eq:critical-two-sided}
 \frac{c_1}{T\log(R_*/r)}\leq\mathcal C_T(r)
       \leq\frac{c_2}{T\log(R_*/r)}.
\end{equation}
The upper bound is realized up to its constant by paths with $A_t=A_0$
over $M\setminus B_{R_*}$.
The lower bound holds for every smooth path in $\mathcal A_u$ with
the stated endpoints, including paths that lose contactness before $T$.

The paths realizing the upper bound also satisfy, for a constant $C$
depending only on the same fixed data,
\begin{equation}\label{eq:critical-uniform}
 \sup_{0\leq t\leq T}
 \left(\|A_t-A_0\|_{H^{n+1}}^2+
       \|F_{A_t}-F_{A_0}\|_{H^n}^2\right)
       \leq\frac{C}{\log(R_*/r)}.
\end{equation}
Consequently, the locus of connections $A\in\mathcal A_u$ for which
$B(u,A)$ is contact is not open in the relative $H^s$ topology for any
$s\leq n+1$, including the critical exponent. It is open for every
$s>n+1$.
\end{theorem}

\begin{proof}
\smallskip\noindent
\emph{Construction.}
Choose a smooth nondecreasing function $\chi:\R\to[0,1]$ such that
$\chi=0$ on $(-\infty,0]$, $\chi=1$ on $[1,\infty)$, and $0<\chi<1$ on $(0,1)$.
For $L\geq1$ and $\rho>0$, set
\begin{equation}\label{eq:log-cutoff}
 r_L=R_*e^{-L},\qquad
 \chi_L(\rho)=\chi\!\left(\frac{\log(R_*/\rho)}{L}\right).
\end{equation}
Extend $\chi_L$ by one at the origin. It is smooth, nonincreasing, equals
one for $\rho\leq r_L$, and vanishes for $\rho\geq R_*$. Put
\[
 b_L=\chi_L\lambda_{\mathrm{std}}
\]
on the chart and extend it by zero. The flatness of $\chi$ at its endpoints
makes this a smooth global one-form. Define
\begin{equation}\label{eq:critical-path}
 A_{L,t}=A_0-\frac{t}{T}\frac{u}{c_u}\pi^*b_L,
 \qquad 0\leq t\leq T.
\end{equation}
Lemma~\ref{lem:scalar-source} places the entire path in $\mathcal A_u$.
In the chart its paired curvature is
\[
 \beta_{L,t}=\dd(\psi_{L,t}\lambda_{\mathrm{std}}),\qquad
 \psi_{L,t}=1-\frac{t}{T}\chi_L.
\]
Expansion of the exterior power gives
\begin{equation}\label{eq:radial-volume}
 \beta_{L,t}^{\,n}
 =\psi_{L,t}^{\,n-1}
   \left(\psi_{L,t}+\frac{\rho}{2}\partial_\rho \psi_{L,t}\right)
      \omega_0^n.
\end{equation}
Indeed,
$\lambda_{\mathrm{std}}=\iota_{Z_0}\omega_0$, where
$Z_0=\frac12\sum_j(x_j\partial_{x_j}+y_j\partial_{y_j})$, and
$\dd\psi(Z_0)=(\rho/2)\psi'$. For $t<T$, both factors in
\eqref{eq:radial-volume} are positive, since $\psi_{L,t}>0$ and
$\partial_\rho \psi_{L,t}\geq0$. Thus $\beta_{L,t}$ is symplectic and
Corollary~\ref{cor:regimes} gives contactness. At $t=T$ it vanishes
identically on $B_{r_L}$. The one-form $\alpha_{L,T}$ remains nowhere zero:
its value on the vertical vector $u^\#$ is still $c_u$. Its Levi rank on
that ball is exactly $\dim G-1$, by Theorem~\ref{thm:rank}.

\smallskip\noindent
\emph{The critical estimate.}
For $r_L<\rho<R_*$ and every integer $j\geq1$, differentiation of
\eqref{eq:log-cutoff} gives
\begin{equation}\label{eq:cutoff-derivatives}
 |\partial^j\chi_L|\leq\frac{C_j}{L\rho^j},\qquad
 |\chi_L|\leq\frac{C\log(R_*/\rho)}{L}.
\end{equation}
Here $\partial^j$ denotes coordinate derivatives of total order $j$.
Every differentiated term contains at least one derivative of $\chi$ and hence
one factor $L^{-1}$. Since the coefficients of $\lambda_{\mathrm{std}}$ are linear,
\begin{equation}\label{eq:primitive-derivatives}
 |\partial^k b_L|\leq\frac{C_k}{L\rho^{k-1}}
 \qquad(k\geq2)
\end{equation}
on the annulus. In dimension $2n$, the top derivative therefore obeys
\begin{align}
 \int_{r_L<\rho<R_*}|\partial^{n+1}b_L|^2\,\dd x
 &\leq\frac{C}{L^2}
       \int_{r_L}^{R_*}\rho^{2n-1-2n}\,\dd\rho \notag\\
 &=\frac{C}{L^2}\log(R_*/r_L)=\frac{C}{L}.
 \label{eq:critical-top-derivative}
\end{align}
The derivatives of orders $2$ through $n$ contribute $O(L^{-2})$.
For orders zero and one, use the second bound in
\eqref{eq:cutoff-derivatives} and the finiteness of
\[
 \int_0^{R_*}\log(R_*/\rho)^2\rho^{2n-1}\,\dd\rho
\]
and of the integral with two additional powers of $\rho$. The contribution
of the inner ball is exponentially small at these orders, and all its
derivatives of order at least two vanish because $b_L=\lambda_{\mathrm{std}}$ there.
Equivalence of coordinate and intrinsic norms on this fixed chart yields
\begin{equation}\label{eq:critical-b-norm}
 \|b_L\|_{H^{n+1}}^2+\|\dd b_L\|_{H^n}^2\leq\frac{C}{L}.
\end{equation}
Multiplication by the fixed smooth section $u/c_u$ is bounded in these norms.
Equations~\eqref{eq:critical-path} and
\eqref{eq:scalar-source-curvature} prove
\eqref{eq:critical-uniform}, and
\[
 \int_0^T\|\dot A_{L,t}\|_{H^{n+1}}^2\,\dd t
 \leq\frac{C}{TL}.
\]
Choosing $L=\log(R_*/r)$ proves the upper bound in
\eqref{eq:critical-two-sided}. Since the terminal connections are noncontact
and converge to $A_0$ in $H^{n+1}$, and hence in every lower $H^s$ norm,
this also proves the claimed failure of openness at and below the endpoint.

\smallskip\noindent
\emph{The lower bound.}
Let $A_t$ be any smooth path in $\mathcal A_u$ starting at $A_0$ and
satisfying $\beta_T|_{B_r}=0$. Write
$\delta\beta=\beta_0-\beta_T$, so $\delta\beta=\omega_0$ on $B_r$.
Choose a fixed chart cutoff $\chi_0$ equal to one on $B_{R_*}$ and supported
in the larger Darboux chart. Extend
\[
 w=\frac{\chi_0}{n}\sum_{j=1}^n
                  \delta\beta(\partial_{x_j},\partial_{y_j})
\]
by zero to $\R^{2n}$. Then $w=1$ on $B_r$, and
\begin{equation}\label{eq:scalar-trace-bound}
 \|w\|_{H^n(\R^{2n})}\leq C\|\delta\beta\|_{H^n(M)}.
\end{equation}
Take $\varphi\in C_c^\infty(B_1)$ with $\varphi\geq0$ and
$\int\varphi=1$, and set
$\varphi_r(x)=r^{-2n}\varphi(x/r)$. Then $\int w\varphi_r=1$.
The Fourier formula for the negative Sobolev norm gives, up to its fixed
normalization,
\[
 \|\varphi_r\|_{H^{-n}}^2
 =\int_{\R^{2n}}(1+|\xi|^2)^{-n}
                         |\widehat\varphi(r\xi)|^2\,\dd\xi.
\]
The region $|\xi|\leq1$ contributes a bounded amount. On
$1<|\xi|\leq r^{-1}$ the radial bound is
$C\int_1^{1/r}\rho^{-1}\,\dd\rho$. On $|\xi|>r^{-1}$, rapid decay
$|\widehat\varphi(r\xi)|\leq C_N(r|\xi|)^{-N}$ gives a bound independent
of $r$. Absorbing the fixed chart scale into the constant, for small $r$ we
obtain
\begin{equation}\label{eq:critical-dual-bump}
 \|\varphi_r\|_{H^{-n}}^2\leq C\log(R_*/r).
\end{equation}
Sobolev duality and~\eqref{eq:scalar-trace-bound} imply
\begin{equation}\label{eq:critical-curvature-lower}
 \|\delta\beta\|_{H^n(M)}^2\geq\frac{c}{\log(R_*/r)}.
\end{equation}
By Lemma~\ref{lem:scalar-source},
\[
 \pi^*\delta\beta
 =-\dd\!\left(\int_0^T B(u,\dot A_t)\,\dd t\right).
\]
The expression in parentheses is basic. Boundedness of this fixed
first-order operator and Cauchy--Schwarz in time give
\[
 \|\delta\beta\|_{H^n}^2
 \leq CT\int_0^T\|\dot A_t\|_{H^{n+1}}^2\,\dd t.
\]
Combining this with~\eqref{eq:critical-curvature-lower} proves the lower
bound. The argument applies to every path with the prescribed terminal curvature.

\smallskip\noindent
\emph{Supercritical openness.}
Let $s>n+1$ and $A'=A_0+a\in\mathcal A_u$. Its paired curvature satisfies
$\pi^*(\beta'-\beta_0)=\dd B(u,a)$. Sobolev embedding on the
$2n$-dimensional base yields
\[
 \|\beta'-\beta_0\|_{C^0}\leq C_s\|a\|_{H^s}.
\]
The minimum singular value of
$\beta_0^\flat:TM\to T^*M$ has a positive lower bound on compact $M$.
If the displayed norm is smaller than that bound, $\beta'$ is nondegenerate.
The centralizer of the unchanged field $u$ is still one-dimensional, so
Corollary~\ref{cor:regimes} makes $B(u,A')$ contact. The same argument at
each connection in this locus proves openness.
\end{proof}

The paths in Theorem~\ref{thm:critical} are smooth on $[0,T]$.
Their transport fields have the following lower bound. For this statement,
assume $g_M$ was chosen Euclidean in the Darboux chart, and use the operator
norm for $\nabla v$.

\begin{proposition}[Necessary transport strain]\label{prop:strain}
For the paths~\eqref{eq:critical-path}, every $t<T$ admits a smooth
$G$-invariant vector field $W_{L,t}$ on $P$ satisfying
\begin{equation}\label{eq:critical-transport}
 \dot A_{L,t}=-\iota_{W_{L,t}}F_{A_{L,t}}.
\end{equation}
An explicit choice is the horizontal lift of
\begin{equation}\label{eq:radial-transport}
 v_{L,t}=
 \frac{\chi_L}
 {T\left(\psi_{L,t}+(\rho/2)\partial_\rho \psi_{L,t}\right)}\,Z_0
\end{equation}
on the chart, extended by zero.

Every smooth vector field $w_t$ on $M$ satisfying
$\partial_t\beta_{L,t}+\Lie_{w_t}\beta_{L,t}=0$ obeys
\begin{equation}\label{eq:strain-blowup}
 \|\nabla w_t\|_{L^\infty}\geq\frac{1}{2(T-t)}.
\end{equation}
In particular, no transport representation of this path extends to $T$ with
integrable Lipschitz norm. The field~\eqref{eq:radial-transport} has
$\nabla v_{L,t}=I/[2(T-t)]$ on $B_{r_L}$.
\end{proposition}

\begin{proof}
For $t<T$, the denominator in~\eqref{eq:radial-transport} is positive.
The flat outer cutoff makes the zero extension smooth. Direct contraction
gives
\[
 \iota_{Z_0}\dd(\psi\lambda_{\mathrm{std}})
       =\left(\psi+\frac\rho2\psi'\right)\lambda_{\mathrm{std}},
 \qquad
 \iota_{v_{L,t}}\beta_{L,t}=\frac1T b_L.
\]
It follows that
$\partial_t\beta_{L,t}+\Lie_{v_{L,t}}\beta_{L,t}=0$.
Lemma~\ref{lem:scalar-source} gives
$F_{A_{L,t}}=(u/c_u)\pi^*\beta_{L,t}$. Its contraction with the horizontal
lift $W_{L,t}$ is $\frac{u}{c_u T}\pi^*b_L=-\dot A_{L,t}$, proving
\eqref{eq:critical-transport} on $P$.

On $B_{r_L}$ we have
$\beta_{L,t}=(1-t/T)\omega_0$. Hence every alternative transport field
$w_t$ satisfies, pointwise on this ball,
\[
 \Lie_{w_t}\omega_0=\frac{1}{T-t}\omega_0,
 \qquad
 \operatorname{div}_{\omega_0^n}w_t=\frac{n}{T-t}.
\]
Since $\omega_0^n$ is a constant multiple of Euclidean volume,
$|\operatorname{tr}(\nabla w_t)|\leq2n\|\nabla w_t\|_{\mathrm{op}}$
proves~\eqref{eq:strain-blowup}. For the radial choice,
$v_{L,t}=Z_0/(T-t)$ on the inner ball, yielding the last assertion.
\end{proof}

\begin{corollary}
For every $T>0$ and $\varepsilon>0$, there is a smooth path in
$\mathcal A_u$ that starts at $A_0$, remains contact for $t<T$, has a
noncontact terminal form, and satisfies
\[
 \int_0^T\|\dot A_t\|_{H^{n+1}}^2\,\dd t<\varepsilon.
\]
\end{corollary}

\begin{proof}
Choose $r$ sufficiently small in the upper bound of
Theorem~\ref{thm:critical}.
\end{proof}

The logarithmic cutoff and the dual estimate are instances of critical
Sobolev capacity; see~\cite{AH}. The construction preserves the parallel
field, its stabilizer reduction, and the class $[\beta_t]=[\beta_0]$.
Proposition~\ref{prop:strain} quantifies the strain required to transport
this localized degeneration.

\begin{remark}[First degeneration of the top exterior power]
Suppose a smooth family $\beta_t$ is symplectic for $t<T$, with
$\beta_t^n=f_t\nu$ for a fixed positive volume form $\nu$ and $f_t>0$.
Continuity gives $f_T\geq0$. Hence $\dd f_T=0$ at every zero of $f_T$.
In particular, a first degeneration attained by a smooth family has no
transverse zero of its top exterior power. For the family above, the
terminal paired curvature vanishes on the whole ball $B_r$.
\end{remark}

\section{Existence, normalization, and holonomy}
\label{sec:holonomy}

For a fixed connection, existence is governed by its full holonomy group,
including holonomy around noncontractible loops.

\begin{proposition}[Holonomy criterion]\label{prop:hol}
Fix $p_0\in P$. Evaluation at $p_0$ is a linear isomorphism
\begin{equation}\label{eq:hol}
 \{u\in\Gamma(\operatorname{ad}P):\dd_Au=0\}
 \longrightarrow
 \g^{\Hol_{p_0}(A)}
 =\{\mu:\Ad_h\mu=\mu\text{ for every }h\in\Hol_{p_0}(A)\}.
\end{equation}
Thus a prescribed value $u(p_0)$ determines at most one parallel field.
If $u$ and $v$ are nonzero parallel fields for the same $A$ and
$\ker B(u,A)=\ker B(v,A)$, then $v=cu$ for a nonzero constant $c$.
\end{proposition}

\begin{proof}
A parallel field is constant along horizontal curves. Transporting $p_0$
around a loop ends at $p_0h$, so equivariance forces
$\Ad_{h^{-1}}u(p_0)=u(p_0)$. Conversely, a holonomy-fixed $\mu$ can be
transported from $p_0$ along any base path and extended equivariantly along
the final fiber. Two choices differ by holonomy and give the same value.
Local smooth dependence of parallel transport proves smoothness. This also
proves uniqueness and~\eqref{eq:hol}.

For the final assertion, restrict the two equal kernels to the vertical
space. Their nonzero linear functionals on $\g$ have the same kernel,
so $v(p)=c(p)u(p)$. The formula
$c=B(v,u)/B(u,u)$ makes $c$ smooth. Parallelism gives
$0=\dd_Av=(\dd c)u$, hence $c$ is constant.
\end{proof}

Unit normalization and a prescribed coorientation remove the scalar
ambiguity. Applying $\dd_A$ to~\eqref{eq:parallel} gives $[F_A,u]=0$,
a local curvature condition. The following example exhibits the additional
constraint imposed by holonomy around noncontractible loops.

\begin{example}[A flat obstruction on a parallelizable base]\label{ex:flat}
Take $G=\SO(3)$ and $M=T^2$. The commuting rotations
\[
 R_1=\operatorname{diag}(1,-1,-1),\qquad
 R_2=\operatorname{diag}(-1,1,-1)
\]
define a representation $\rho:\Z^2\to\SO(3)$ by
$\rho(m,n)=R_1^mR_2^n$. Form the bundle
\[
 P=(\R^2\times\SO(3))/\Z^2,\qquad
 (m,n)\cdot(s,t,g)=(s+m,t+n,\rho(m,n)g).
\]
The product connection descends and is flat. Under
$\mathfrak{so}(3)\cong\R^3$, the two rotations have fixed spaces
$\R e_1$ and $\R e_2$, with zero intersection. Proposition~\ref{prop:hol}
implies that this connection admits no nonzero parallel adjoint field.
Thus this flat connection on a parallelizable base has trivial holonomy-fixed
subspace in the adjoint representation.
\end{example}

On a closed base, the holonomy criterion also has a spectral formulation.

\begin{proposition}[Normalized variational criterion]\label{prop:variation}
Suppose $M$ is closed and equip it with a Riemannian metric. For a fixed smooth
connection $A$, set
\begin{equation}\label{eq:rayleigh}
 \lambda_{\min}(A)=\inf_{s\in H^1(\operatorname{ad}P),\ \|s\|_{L^2}=1}
                    \|\nabla_A s\|_{L^2}^2.
\end{equation}
The infimum is attained by a smooth eigensection of
$\nabla_A^*\nabla_A$. Moreover,
\[
 \lambda_{\min}(A)=0
 \quad\Longleftrightarrow\quad
 \g^{\Hol_{p_0}(A)}\ne0.
\]
If the fixed space is zero, then $\lambda_{\min}(A)>0$.
\end{proposition}

\begin{proof}
For a fixed smooth connection on a compact manifold,
$\|s\|_{H^1}$ is controlled by
$\|\nabla_As\|_{L^2}+\|s\|_{L^2}$. A minimizing sequence is therefore
bounded in $H^1$. Weak $H^1$ compactness and strong $L^2$ compactness give a
normalized minimizer. Weak lower semicontinuity proves minimality, and the
Euler--Lagrange equation is
$\nabla_A^*\nabla_As=\lambda_{\min}(A)s$. Elliptic regularity gives smoothness.
The minimum is zero precisely when the minimizer is parallel; now use
Proposition~\ref{prop:hol}.
\end{proof}

The parallel fields form the zero eigenspace of the connection Laplacian.
Each nonzero element determines its invariant hyperplane distribution by
\eqref{eq:alpha}.

\section{Global realization and characteristic quotients}
\label{sec:global}

Fix a nonzero value $\mu=u(p_0)$, let $H=G_\mu$ be its stabilizer, and let
$\h=\g_\mu$. The level set
\begin{equation}\label{eq:Q}
 Q=u^{-1}(\mu)
\end{equation}
is a principal $H$-subbundle over $M$. To see this directly, horizontally
transport a point of $Q$ over a coordinate neighborhood to obtain a local
section of $Q$. Within each fiber the level set is its right $H$-orbit.
On $TQ$, equation~\eqref{eq:parallel} implies $[A,\mu]=0$; hence the
restriction $A_Q$ is an $\h$-valued connection on $Q$.

Conversely, any connection $A_Q$ on an $H$-reduction extends to a connection
on $P\cong Q\times_H G$. We use the equivalence
$(q,g)\sim(qh,h^{-1}g)$. On $Q\times G$ the extended objects are
\begin{align}
 A&=\Ad_{g^{-1}}A_Q+g^{-1}\dd g,
 &u&=\Ad_{g^{-1}}\mu,\label{eq:extension}\\
 \alpha&=B(\mu,A_Q)+B(\mu,\dd g\,g^{-1}).\label{eq:coupling}
\end{align}
These expressions descend, and substitution gives $\dd_Au=0$. In
particular, the existence problem on a varying connection is precisely the
existence of such a stabilizer reduction.

\subsection{Realization of paired curvature}

Let $\pi_Q:Q\to M$ be a fixed principal $H$-bundle. Since $\mu$ commutes with
$\h$, the linear functional $B(\mu,\cdot)|_\h$ annihilates
$[\h,\h]$. For each connection $A_Q$ define
\[
 \pi_Q^*\beta_{A_Q}=B(\mu,F_{A_Q}).
\]

\begin{theorem}[Curvature realization on a fixed reduction]\label{thm:realize}
The class
\begin{equation}\label{eq:cclass}
 c_\mu(Q)=[\beta_{A_Q}]\in H^2_{\mathrm{dR}}(M)
\end{equation}
is independent of $A_Q$. Every closed representative of this class is realized
by a connection on $Q$. More explicitly, if
$\beta'=\beta_{A_Q}+\dd\gamma$, put
\begin{equation}\label{eq:realize}
 z=\frac{\mu}{B(\mu,\mu)},\qquad
 A'_Q=A_Q+z\,\pi_Q^*\gamma.
\end{equation}
Then $A'_Q$ is a connection and $\beta_{A'_Q}=\beta'$.
Consequently, a fixed reduction with $\dim H=1$ yields a contact form of
the type~\eqref{eq:alpha} exactly when $c_\mu(Q)$ has a symplectic
representative.
\end{theorem}

\begin{proof}
For another connection $A_Q+a$, the difference $a$ is horizontal and
$H$-equivariant. Thus $B(\mu,a)=\pi_Q^*\gamma$ for a base one-form $\gamma$.
Using the curvature difference formula and the fact that $\mu$ annihilates
brackets in $\h$ gives
\[
 B(\mu,F_{A_Q+a}-F_{A_Q})
 =\dd B(\mu,a)=\pi_Q^*\dd\gamma.
\]
This proves independence. The element $z$ is central in $\h$ and fixed by
$H$, so the addition in~\eqref{eq:realize} is tensorial and defines a
connection. Its curvature changes by $z\,\pi_Q^*\dd\gamma$, whose pairing with
$\mu$ is $\pi_Q^*\dd\gamma$. The contact assertion now follows from
Corollary~\ref{cor:regimes}.
\end{proof}

The class $c_\mu(Q)$ is the degree-two characteristic class associated with the
invariant linear functional $B(\mu,\cdot)|_\h$. A character of $H$ whose differential is a normalized multiple of
$B(\mu,\cdot)$ determines an integral normalization of this class. On a closed base of
dimension $2n>0$, a symplectic representative requires
$c_\mu(Q)^n\ne0$. Formula~\eqref{eq:realize} realizes every symplectic representative
by a connection on the fixed reduction.

On the reduction, $\pi_Q^*\beta=\dd B(\mu,A_Q)$.
On $P$, the differential of $\alpha$ is given by~\eqref{eq:dalpha}.

\subsection{Characteristic quotients}

Assume for the rest of this subsection that $\beta$ is symplectic. Put
\[
 \mathfrak k=\h\cap\mu^\perp,
\]
and let $K\subset H$ be its connected analytic subgroup. The subspace
$\mathfrak k$ is an ideal of $\h$ and invariant under $\Ad_H$, since
$\mu$ is fixed by $H$. Thus $K$ is a normal subgroup, although it may not
be closed.

\begin{theorem}[Characteristic quotient]\label{thm:quotient}
The characteristic leaves in $P\cong Q\times_HG$ are the immersed
submanifolds
\[
 \{[q,kg]:k\in K\}.
\]
If $K$ is closed, their leaf space is the smooth bundle
\begin{equation}\label{eq:quotient}
 \overline P=Q\times_H(K\backslash G),
\end{equation}
and $\alpha$ descends to a contact form on $\overline P$.
If $K$ is not closed, the leaf space is not $T_1$ and hence is not a
Hausdorff manifold.
\end{theorem}

\begin{proof}
Formula~\eqref{eq:char} and nondegeneracy of $\beta$ identify the
characteristic space at $[q,g]$ with the tangent space to the left coset
$Kg$. On $G$, the form $B(\mu,\dd g\,g^{-1})$ is invariant under left
translation by $K$ and annihilates its tangent directions; its exterior
derivative does so as well. The form in~\eqref{eq:coupling} is therefore
basic for these leaves.

If $K$ is closed, $K\backslash G$ is smooth. Normality in $H$ defines
the required left $H$-action on $K\backslash G$, and local trivializations
of $Q$ identify~\eqref{eq:quotient} with the leaf quotient. The descended
form is nowhere zero. Its Levi form is obtained from $\dd\alpha|_\D$
by dividing out its entire radical, and is nondegenerate. This is a contact
form. If $K$ is not closed, each such leaf is nonclosed in its fiber,
being dense in the larger coset $\overline K g$. A $T_1$ leaf space would
make the inverse image of every singleton closed, which is impossible.
\end{proof}

The smooth quotient is a homogeneous contact coupling construction. The
closedness condition is a global condition on the subgroup; it is invisible
to the pointwise rank formula.

\begin{example}[An irrational characteristic foliation]\label{ex:irrational}
On $P=\R^2\times T^2$, take angular coordinates of period $2\pi$ and set
\[
 A=(\dd\theta_1+x\,\dd y,\dd\theta_2),\qquad
 u=(1,\sqrt2).
\]
The group is abelian, so $u$ is parallel, and
\[
 \alpha=\dd\theta_1+\sqrt2\,\dd\theta_2+x\,\dd y,
 \qquad\beta=\dd x\wedge\dd y.
\]
Its characteristic direction is
$-\sqrt2\,\partial_{\theta_1}+\partial_{\theta_2}$, with dense leaves in
each torus fiber. Replacing $\sqrt2$ by a rational number gives closed
circle leaves and a smooth contact quotient. Both examples have the same
pointwise Levi rank.
\end{example}

\section{Examples and a topological obstruction}
\label{sec:examples}

\subsection{Flat nonabelian fibers and a curved circle bundle}

For $P=G\to\{\mathrm{pt}\}$ take
$A=g^{-1}\dd g$ and $u(g)=\Ad_{g^{-1}}\mu$. Then $F_A=0$ and
$\dd_Au=0$, but
\[
 \alpha=B(\mu,\dd g\,g^{-1}).
\]
For $G=\SU(2)$ and $\mu\ne0$, its Levi form has rank two and $\alpha$
is contact. For $G=\SU(3)$ and regular $\mu$, the centralizer has
dimension two: the Levi rank is six and the characteristic dimension is
one. These examples isolate the vertical bracket contribution.

For the complementary case, on $\R^2\times S^1$ use the real convention
$A=\dd\theta+x\,\dd y$ and $u=1$. Here $\alpha=A$,
$\beta=\dd x\wedge\dd y$, and
$\alpha\wedge\dd\alpha=\dd\theta\wedge\dd x\wedge\dd y$.
Thus $\alpha$ is contact despite the identically vanishing coadjoint action
of the curvature. Setting $A=\dd\theta$ instead gives an integrable
distribution. These two circle examples distinguish $[F_A,u]=0$ from
$B(u,F_A)=0$.

\subsection{The \texorpdfstring{$\SU(2)$}{SU(2)} case}

Let $G=\SU(2)$, take
$T=\operatorname{diag}(i,-i)$, and normalize $B(T,T)=1$.
After conjugation write $\mu=\ell T$ with $\ell>0$. Its stabilizer is
\[
 H=\{\operatorname{diag}(e^{i\theta},e^{-i\theta}):\theta\in\R\}.
\]
A reduction $Q$ determines the usual splitting of the associated
rank-two complex vector bundle
\begin{equation}\label{eq:chern}
 E\cong L\oplus L^{-1},\qquad c_2(E)=-c_1(L)^2.
\end{equation}
The second identity follows by multiplying the total Chern classes
$(1+c_1(L))(1-c_1(L))$.

Write $A_Q=T\vartheta$, where $\vartheta$ is a real circle connection
with value $1$ on the period-$2\pi$ generator. Put
$\dd\vartheta=\pi_Q^*f$. Choose one of the two associated-line conventions
so that $c_1(L)_{\R}=[f/(2\pi)]$; interchanging $L$ and $L^{-1}$
reverses this sign and leaves the following square unchanged. Then
\begin{equation}\label{eq:su2beta}
 \beta=\ell f,\qquad
 \left[\frac{\beta}{2\pi\ell}\right]^2=-c_2(E)_{\R}.
\end{equation}

\begin{corollary}\label{cor:su2}
Let $M$ be a closed connected four-manifold. If a parallel pair on a
principal $\SU(2)$-bundle produces a contact form $B(u,A)$, then
$c_2(E)_{\R}\ne0$. In the orientation determined by the symplectic form
$\beta$, one has
\begin{equation}\label{eq:chernnumber}
 \langle c_2(E),[M]\rangle
 =-\frac{1}{(2\pi\ell)^2}\int_M\beta\wedge\beta<0.
\end{equation}
Conversely, a circle reduction whose real first Chern class has a
symplectic representative supplies such a contact form by
Theorem~\ref{thm:realize}.
\end{corollary}

On $S^4$, any nonzero
$c_2(E)$ excludes a parallel nonzero adjoint field for every connection,
because $H^2(S^4;\Z)=0$ makes~\eqref{eq:chern} impossible.

For a compact realization, take a line bundle over $\mathbb{CP}^2$ whose
first Chern class is the positive generator and choose a circle connection
with $f/(2\pi)$ representing the Fubini--Study class. Extending its circle
bundle to $\SU(2)$ gives $E=L\oplus L^{-1}$ and a contact form on the
seven-dimensional total space. The rank formula gives four horizontal
and two vertical Levi directions. This is a contact coupling example
within the classical construction of~\cite{Lerman}.

\section{Evolution and its mixed-curvature obstruction}
\label{sec:dynamics}

Time dependence requires compatibility in both space and time. Let $A_t$ be
a smooth family of connections and let $\eta_t:P\to\g$ be equivariant.
For an adjoint-valued time-dependent form define
\[
 D_t=\partial_t+[\eta_t,\cdot],\qquad
 E_t=\dot A_t-\dd_{A_t}\eta_t.
\]
The form $E_t$ is horizontal and equivariant. It is the mixed curvature of
the connection $\mathbb A=A_t+\eta_t\,\dd t$ on $P\times I\to M\times I$:
\begin{equation}\label{eq:spacetime}
 F_{\mathbb A}=F_{A_t}+\dd t\wedge E_t.
\end{equation}

\begin{proposition}[Compatibility in space and time]\label{prop:mixed}
Let $u_t$ solve $D_tu_t=0$ and suppose $\dd_{A_0}u_0=0$. Then
\begin{equation}\label{eq:commute}
 D_t(\dd_{A_t}u_t)=[E_t,u_t].
\end{equation}
Consequently $\dd_{A_t}u_t=0$ for all $t$ if and only if
$[E_t,u_t]=0$ for all $t$. When these conditions hold, the base forms
defined by
\[
 \pi^*\beta_t=B(u_t,F_{A_t}),\qquad
 \pi^*e_t=B(u_t,E_t)
\]
satisfy
\begin{equation}\label{eq:faraday}
 \dd\beta_t=0,\qquad \partial_t\beta_t=\dd e_t.
\end{equation}
In particular $[\beta_t]$ is independent of time.
\end{proposition}

\begin{proof}
Expanding the covariant derivatives gives the general commutator identity
\[
 D_t(\dd_{A_t}u_t)=\dd_{A_t}(D_tu_t)+[E_t,u_t].
\]
This proves~\eqref{eq:commute}. If the right side vanishes, the
initially zero form $\dd_{A_t}u_t$ solves a homogeneous linear temporal
equation and remains zero. Necessity follows from the same identity.
Under compatibility, $u_t$ is parallel for $\mathbb A$; invariant pairing
with its Bianchi identity gives
\[
 \dd_{M\times I}(\beta_t+\dd t\wedge e_t)=0.
\]
Separating spatial and temporal degrees proves~\eqref{eq:faraday}.
\end{proof}

\begin{example}[Spatial flatness does not ensure temporal compatibility]
On the trivial $\SU(2)$-bundle over $\R$, use a local gauge with
$A_t=tT_1\,\dd x$, $\eta_t=0$, and $u_0=T_3$, where
$[T_1,T_3]\ne0$. Spatial curvature is zero for every $t$.
The temporal equation fixes $u_t=T_3$, whereas
$\dd_{A_t}u_t=t[T_1,T_3]\,\dd x$. The obstruction is exactly
$[E_t,u_t]=[T_1,T_3]\,\dd x$.
\end{example}

The following equation has a stronger transport property. Here $W_t$ is a
$G$-invariant vector field on $P$ projecting to $v_t$ on $M$:
\begin{equation}\label{eq:transport}
 \dot A_t=\dd_{A_t}\eta_t-\iota_{W_t}F_{A_t}.
\end{equation}
The vector field $W_t$ specifies the infinitesimal bundle transport.

\begin{theorem}[Transport preserves the Levi geometry]\label{thm:transport}
Suppose~\eqref{eq:transport} holds, and let $u_t$ solve
$D_tu_t=0$ from nonzero parallel initial data. Put
\begin{equation}\label{eq:Y}
 Z_t=W_t-(\eta_t+A_t(W_t))^\#.
\end{equation}
Assume the flow $\Phi_t$ of $Z_t$, with $\Phi_0=\mathrm{id}$, is
defined as a family of principal-bundle automorphisms on the interval
considered. Let $\phi_t$ be its base flow. Then $u_t$ is parallel for $A_t$, and
\begin{equation}\label{eq:pullbacks}
 \Phi_t^*A_t=A_0,\quad \Phi_t^*u_t=u_0,\quad
 \Phi_t^*\alpha_t=\alpha_0,\quad \phi_t^*\beta_t=\beta_0.
\end{equation}
Consequently
\begin{equation}\label{eq:ranktransport}
 \rank(\dd\alpha_t|_{\D_t})_{\Phi_t(p)}
 =\rank(\dd\alpha_0|_{\D_0})_p.
\end{equation}
Integrability, contactness, and every Levi rank locus are transported by the
flow.
\end{theorem}

\begin{proof}
For any vector field $Y$, Cartan's formula gives
\begin{equation}\label{eq:cartan}
 \Lie_YA=\iota_YF_A+\dd_A(A(Y)).
\end{equation}
In~\eqref{eq:Y}, one has $A_t(Z_t)=-\eta_t$, and horizontality of
$F_{A_t}$ gives $\iota_{Z_t}F_{A_t}=\iota_{W_t}F_{A_t}$.
Thus~\eqref{eq:transport} is
$\dot A_t+\Lie_{Z_t}A_t=0$. Differentiating the pullback along the flow
proves $\Phi_t^*A_t=A_0$.

Define $\widetilde u_t=(\Phi_t^{-1})^*u_0$. Naturality of covariant
derivatives shows $\dd_{A_t}\widetilde u_t=0$. Hence
\[
 \Lie_{Z_t}\widetilde u_t
 =-[A_t(Z_t),\widetilde u_t]=[\eta_t,\widetilde u_t].
\]
The transport equation for $\widetilde u_t$ is therefore
$D_t\widetilde u_t=0$. Uniqueness of this pointwise linear ODE proves
$u_t=\widetilde u_t$. Pairing the transported fields gives the third
identity in~\eqref{eq:pullbacks}; curvature naturality gives the fourth.
Exterior differentiation and restriction to the transported kernels yield
\eqref{eq:ranktransport}.
\end{proof}

Equivalently, $E_t=-\iota_{W_t}F_{A_t}$ and
$e_t=-\iota_{v_t}\beta_t$, so~\eqref{eq:faraday} becomes
\[
 \partial_t\beta_t+\Lie_{v_t}\beta_t=0.
\]
This identifies~\eqref{eq:transport} as transport by bundle automorphisms.
On a compact base the smooth invariant vector fields involved have flows
on each compact time interval on which their coefficients are smooth,
since $G$ is compact and therefore $P$ is compact.

General compatible evolutions need not preserve rank. For example,
$A_t=\dd\theta+t x\,\dd y$, $u_t=1$, $\eta_t=0$ on
$\R^2\times S^1$ satisfies Proposition~\ref{prop:mixed} with
$E_t=x\,\dd y$. Its paired curvature is
$\beta_t=t\,\dd x\wedge\dd y$: the distribution is integrable at $t=0$
and contact when $t\ne0$. Such a family cannot satisfy a smooth transport
law~\eqref{eq:transport} through $t=0$. Constancy of a de Rham class alone
does not preserve pointwise rank.

\subsection{Curvature forces and power}

For the magnetic force convention below, the paired spacetime curvature
separates the gyroscopic force and temporal power. If $\beta_t$ acts on a velocity $w$ by
the force covector $\iota_w\beta_t$, then its instantaneous power is
\[
 (\iota_w\beta_t)(w)=\beta_t(w,w)=0.
\]
For a virtual displacement $\delta q$, the value
$\beta_t(w,\delta q)$ can be nonzero; this is a virtual-work pairing.
On spacetime, the closed form
$\mathcal B=\beta_t+\dd t\wedge e_t$ gives
\[
 \mathcal B(\partial_t+w,\delta q)
 =\beta_t(w,\delta q)+e_t(\delta q).
\]
With this force convention, the power is $e_t(w)$.

\section{Adapted metrics for invariant contact forms}
\label{sec:metrics}

The contact paths above on circle bundles over surfaces admit invariant
adapted metrics. We first treat an arbitrary invariant contact form on a
three-manifold and then apply the construction to these paths.

Throughout this section, $N$ is a closed three-manifold, $\alpha$ is a
contact form, and $X$ is the nowhere-zero generator of a locally free
$S^1$-action preserving $\alpha$. We orient $N$ by
$\alpha\wedge\dd\alpha$. The action may have exceptional orbits.
Write $\xi=\ker\alpha$, and let $R_\alpha$ be the Reeb field, determined by
\[
 \alpha(R_\alpha)=1,\qquad \iota_{R_\alpha}\dd\alpha=0.
\]
The projection onto $\xi$ along $R_\alpha$ is
$\Pi_\alpha V=V-\alpha(V)R_\alpha$.

\begin{definition}
A Riemannian metric $g$ is adapted to $\alpha$ if
$*_g\dd\alpha=\kappa\alpha$ for a smooth positive function $\kappa$.
When $\kappa$ is constant, $\alpha$ is a curl eigenform.
Here $*_g$ is the Hodge star for $g$ and the contact orientation.
For a vector field $V$, the curl convention is
$(\operatorname{curl}_g V)^\flat=*_g\dd(V^\flat)$.
\end{definition}

This definition agrees with~\cite[Definition 2.1]{Kom} and allows the
norm of $\alpha$ to vary. The simultaneous normalization problem fixes
the eigenvalue and the length of $X$.

\subsection{The metric construction}

Decompose the circle generator as
\begin{equation}\label{eq:circle-split}
 X=fR_\alpha+Y,\qquad f=\alpha(X),\qquad Y=\Pi_\alpha X.
\end{equation}
The action preserves $R_\alpha$, $\xi$, $\Pi_\alpha$, $f$, and $Y$.
Choose an invariant metric $h$ on $\xi$ with area form
$\dd\alpha|_\xi$. Such a metric exists: average any metric on $\xi$ over
the circle, write its area form as $r\dd\alpha|_\xi$ with $r>0$, and
divide the averaged metric by $r$. Extend $h$ to a symmetric tensor on
$TN$ by
\[
 h(V,W)=h(\Pi_\alpha V,\Pi_\alpha W),
\]
so $h(R_\alpha,\cdot)=0$. Put $q=h(Y,Y)$.
Since $X$ is nowhere zero, $f$ and $q$ have no common zero.

\begin{theorem}[Simultaneous metric normalization]\label{thm:metric}
For every constant $\kappa>0$, the formulas
\begin{equation}\label{eq:metric-factor}
 \sigma=\frac{2}{q+\sqrt{q^2+4\kappa^2f^2}},\qquad
 g=\kappa^2\sigma^2\alpha\otimes\alpha+\sigma h
\end{equation}
define a smooth Riemannian metric on $N$ satisfying
\begin{equation}\label{eq:metric-identities}
 *_g\dd\alpha=\kappa\alpha,\qquad
 g(X,X)=1,\qquad \Lie_Xg=0.
\end{equation}
\end{theorem}

\begin{proof}
The denominator and the argument of the square root in
\eqref{eq:metric-factor} are strictly positive. Thus $\sigma$ is smooth
and positive, and the Reeb splitting $TN=\mathbb RR_\alpha\oplus\xi$
makes $g$ positive definite. Every term is invariant, so $\Lie_Xg=0$.

The function $\sigma$ is the unique positive solution of
\begin{equation}\label{eq:metric-quadratic}
 \kappa^2f^2\sigma^2+q\sigma=1.
\end{equation}
Using~\eqref{eq:circle-split}, the left side is $g(X,X)$.
The identity holds at $f=0$ and at $q=0$ by direct substitution.

Take a local positively oriented $h$-orthonormal coframe
$\eta^1,\eta^2$ on $\xi$ and extend it by zero on $R_\alpha$.
The area normalization gives $\dd\alpha=\eta^1\wedge\eta^2$.
A positively oriented $g$-orthonormal coframe is
\[
 e^0=\kappa\sigma\alpha,\qquad
 e^1=\sqrt\sigma\,\eta^1,\qquad
 e^2=\sqrt\sigma\,\eta^2.
\]
Consequently,
\[
 *_g\dd\alpha=\sigma^{-1}*_g(e^1\wedge e^2)
             =\sigma^{-1}e^0=\kappa\alpha.
\qedhere
\]
\end{proof}

Along the locus $f=0$, one has $q>0$ and $\sigma=1/q$.
Where $Y=0$, one has $f\ne0$ and $\sigma=1/(\kappa|f|)$.
Formula~\eqref{eq:metric-factor} gives a single smooth metric across both
loci. Averaging and the metric construction take place on the smooth
total space, including points with finite circle stabilizer.

In the weakly compatible family
$g=a^2\alpha^2+(a/\kappa)h$ of~\cite[Proposition 2.1(3)]{EKM}, the
choice $a=\kappa\sigma$ imposes the unit-length condition through
\eqref{eq:metric-quadratic}. This gives an affirmative answer to
\cite[Question 2.7]{Kom}.

The formulation also covers a nowhere-zero strict contact field tangent
to the fibers of a Seifert fibration. On a regular fiber chart, its
positive period $\tau$ is smooth. Differentiating
$p\mapsto\varphi_{\tau(p)}(p)=p$ and applying the invariant form $\alpha$
gives $f\dd\tau=0$. At $f=0$, Cartan's formula yields
$\dd f=-\iota_X\dd\alpha\ne0$, so $f\ne0$ is dense and
$\dd\tau=0$. The regular orbit space of a connected Seifert manifold
is connected. Hence the principal period is constant, and its time map
is the identity everywhere by continuity. The field therefore generates
a locally free circle action, to which Theorem~\ref{thm:metric} applies.

\subsection{The space of normalized metrics}

For fixed $\alpha$, $X$, and $\kappa>0$, let $\mathcal M_\xi$ denote the
space of invariant metrics on $\xi$ with area form $\dd\alpha|_\xi$.
Let $\mathcal G_\kappa$ denote the space of invariant Riemannian metrics
on $N$ satisfying $*_g\dd\alpha=\kappa\alpha$ and $g(X,X)=1$.
Both spaces carry their $C^\infty$ topologies.

\begin{proposition}\label{prop:metric-space}
The construction~\eqref{eq:metric-factor} defines a homeomorphism
$\mathcal M_\xi\to\mathcal G_\kappa$. The map and its inverse preserve
smooth parameter families. In particular, $\mathcal G_\kappa$ is
nonempty and contractible.
\end{proposition}

\begin{proof}
Given $g\in\mathcal G_\kappa$, put $a=\sqrt{g(R_\alpha,R_\alpha)}$.
The identity $\dd\alpha=\kappa *_g\alpha$ implies that
$\alpha^{\sharp_g}$ spans $\ker\dd\alpha=\mathbb RR_\alpha$.
Thus $R_\alpha\perp_g\xi$, and restriction to $\xi$ gives
\[
 \dd\alpha|_\xi=\frac\kappa a\operatorname{area}_{g|_\xi}.
\]
It follows that
\begin{equation}\label{eq:metric-inverse}
 h=\frac\kappa a\,g|_\xi
\end{equation}
belongs to $\mathcal M_\xi$. With $\sigma=a/\kappa$, the equation
$g(X,X)=1$ becomes~\eqref{eq:metric-quadratic}, which determines
$\sigma$ uniquely. This proves the inverse formula. Both constructions
depend smoothly on their tensor arguments.

Fix $h_*\in\mathcal M_\xi$. For $h\in\mathcal M_\xi$, set
$k_s=(1-s)h+sh_*$ for $0\leq s\leq1$ and write
$\operatorname{area}_{k_s}=r_s\dd\alpha|_\xi$.
Then $h_s=k_s/r_s$ is an invariant metric of the required area.
The path contracts $\mathcal M_\xi$ to $h_*$ and fixes $h_*$
throughout. Transporting this contraction by
\eqref{eq:metric-factor} and~\eqref{eq:metric-inverse} proves the result.
\end{proof}

\begin{corollary}
For the metric~\eqref{eq:metric-factor}, the field
$V=\alpha^{\sharp_g}=R_\alpha/(\kappa^2\sigma^2)$ satisfies
\[
 \operatorname{curl}_gV=\kappa V,\qquad
 \operatorname{div}_gV=0,\qquad [X,V]=0.
\]
The orbits of $X$ are unit-speed geodesics.
\end{corollary}

\begin{proof}
The curl identity follows from~\eqref{eq:metric-identities}.
Since $\kappa$ is constant, $\dd^2\alpha=\kappa\dd *_g\alpha=0$
gives the divergence identity. Invariance gives $[X,V]=0$.
Finally, a constant-length Killing field satisfies
$\nabla_XX=-\frac12\operatorname{grad}|X|^2=0$.
\end{proof}

\subsection{Circle bundles and examples}

Suppose the circle action is free, and write $\pi:N\to\Sigma$ for its
principal bundle over a surface. Let $\vartheta$ be an invariant real
connection with $\vartheta(X)=1$ and curvature
$\dd\vartheta=\pi^*\Omega$. Every invariant one-form has a unique
decomposition
\begin{equation}\label{eq:circle-form}
 \alpha=(\pi^*f)\vartheta+\pi^*\gamma,
 \qquad f\in C^\infty(\Sigma),\quad\gamma\in\Omega^1(\Sigma).
\end{equation}
Here $f$ is the function $\alpha(X)$ on the quotient. Exterior
differentiation gives
\begin{equation}\label{eq:circle-contact}
 \alpha\wedge\dd\alpha
 =\vartheta\wedge\pi^*
       \bigl(f^2\Omega+f\dd\gamma+\gamma\wedge\dd f\bigr).
\end{equation}
Thus $\alpha$ is contact exactly when the two-form in parentheses
is nowhere zero. The parallel circle-bundle case is obtained by taking
$\vartheta=A$, constant nonzero $f$, and $\gamma=0$.
Formula~\eqref{eq:circle-contact} also describes
invariant contact forms with a nonempty dividing locus $f^{-1}(0)$.

For the real circle convention $A(X)=1$, $u=1$, and
$\alpha=A$, one has $R_\alpha=X$, $f=1$, and $q=0$.
Theorem~\ref{thm:metric} reduces to
\begin{equation}\label{eq:circle-normalization}
 g=\alpha^2+\kappa^{-1}h,\qquad
 \operatorname{area}_h=\dd\alpha|_\xi.
\end{equation}

\begin{corollary}[Metric degeneration along a contact path]
\label{cor:metric-degeneration}
Apply Theorem~\ref{thm:critical} to a principal circle bundle over a
closed surface with the normalization $u=1$, $\alpha_t=A_t$.
For every $\kappa>0$, the forms $\alpha_t$, $t<T$, admit a smooth family
of invariant metrics with $*_{g_t}\dd\alpha_t=\kappa\alpha_t$ and
$g_t(X,X)=1$. Every continuous extension to $t=T$ of any such family
as symmetric tensors is degenerate at each point over the collapsing ball.
\end{corollary}

\begin{proof}
Orient $\Sigma$ by $\beta_0$, fix a positive area form $\nu$, and choose
a metric $g_\Sigma$ with area $\nu$. Write $\beta_t=s_t\nu$, where
$s_t>0$ for $t<T$. The tensor
$h_t=(s_t\circ\pi)\pi^*g_\Sigma$ restricts to an invariant metric on
$\xi_t=\ker\alpha_t$ with area form $\dd\alpha_t|_{\xi_t}$.
Equation~\eqref{eq:circle-normalization} gives
\[
 g_t=\alpha_t^2+\kappa^{-1}(s_t\circ\pi)\pi^*g_\Sigma.
\]
This family is smooth for $t<T$.

For any normalized family with a continuous tensor limit, the limit is
positive semidefinite. If it were positive definite at a point over the
collapsing ball, continuity of the Hodge star there would give
$*_{g_T}\dd\alpha_T=\kappa\alpha_T$. On that ball,
$\dd\alpha_T=\pi^*\beta_T=0$, while $\alpha_T(X)=1$.
The displayed identity would therefore imply $0=\kappa$, a contradiction.
\end{proof}

On $\mathbb T^3$ with coordinates of period $2\pi$, take
\[
 \alpha=\cos z\,\dd x-\sin z\,\dd y,\qquad X=\partial_x.
\]
Then $\alpha\wedge\dd\alpha=\dd x\wedge\dd y\wedge\dd z$ and
\[
 R_\alpha=\cos z\,\partial_x-\sin z\,\partial_y,
 \qquad Z=\sin z\,\partial_x+\cos z\,\partial_y.
\]
Choose $h$ making $Z,\partial_z$ orthonormal on $\xi$.
Here $f=\cos z$, $Y=\sin z\,Z$, and $q=\sin^2z$, so
\[
 \sigma_\kappa(z)=\frac{2}{\sin^2z+
                    \sqrt{\sin^4z+4\kappa^2\cos^2z}},
\]
\[
 g_\kappa=\kappa^2\sigma_\kappa^2\alpha^2+
 \sigma_\kappa\bigl((\sin z\,\dd x+\cos z\,\dd y)^2+\dd z^2\bigr).
\]
Both $f=0$ and $Y=0$ occur. The metric is smooth for all $\kappa>0$;
at $\kappa=1$, it is the flat metric.

\end{document}